\documentclass[a4paper]{amsart}
\usepackage{multicol}
\usepackage{amsmath}
\usepackage{amssymb}
\usepackage{amsthm}
\usepackage{enumerate}

\usepackage[abbrev]{amsrefs}
\usepackage[english]{babel}
\usepackage[T1]{fontenc}
\usepackage[utf8]{inputenc}

\newcommand{\F}{\mathcal{F}}
\newcommand{\Sc}{\mathcal{S}}
\newcommand{\R}{\mathbb R}
\newcommand{\T}{\mathbb T}
\newcommand{\N}{\mathbb N}
\newcommand{\Z}{\mathbb Z}
\newcommand{\C}{\mathbb C}

\newcommand{\e}{\varepsilon}
\newcommand{\lb}{\langle}
\newcommand{\rb}{\rangle}

\newcommand{\set}[1]{\left\{#1\right\}}
\newcommand{\RN}[1]{%
	\textup{\uppercase\expandafter{\romannumeral#1}}%
}

\DeclareMathOperator{\supp}{supp}
\newtheorem{lemma}{Lemma}
\newtheorem*{kor}{Corollary}
\newtheorem{theorem}{Theorem}
\newtheorem*{theorem*}{Theorem}
\newtheorem{prop}{Proposition}
\newtheorem*{remark*}{Remark}

\begin{document}

\title[LWP for the Novikov-Veselov equation]{Low regularity local well-posedness for the zero energy Novikov-Veselov equation}

\author[J.~Adams]{Joseph~Adams}

\address{Heinrich-Heine-Universit\"at D\"usseldorf, Mathematisches Institut,
Universit\"atsstra{\ss}e~1, 40225 D\"usseldorf, Germany}
\email{joseph.adams@hhu.de}

\author[A.~Grünrock]{Axel~Grünrock}

\address{Heinrich-Heine-Universit\"at D\"usseldorf, Mathematisches Institut,
Universit\"atsstra{\ss}e~1, 40225 D\"usseldorf, Germany}
\email{gruenroc@math.uni-duesseldorf.de}


\begin{abstract}
  The initial value problem $u(x,y,0)=u_0(x,y)$ for the zero energy Novikov-Veselov equation
  \[
  \partial_tu+(\partial ^3 + \overline{\partial}^3)u +3(\partial (u\overline{\partial}^{-1}\partial u)+\overline{\partial}(u\partial^{-1}\overline{\partial}u))=0
  \]
  is investigated by the Fourier restriction norm method. Local well-posedness is
  shown in the nonperiodic case for $u_0 \in H^s(\R^2)$ with $s > - \frac{3}{4}$
  and in the periodic case for data $u_0 \in H^s_0(\T^2)$ with mean zero, where
  $s > - \frac{1}{5}$. Both results rely on the structure of the nonlinearity,
  which becomes visible with a symmetrization argument. Additionally, for the
  periodic problem a bilinear Strichartz-type estimate is derived.
\end{abstract}

\maketitle

\section{Introduction}

Besides both the Kadomtsev-Petviashvili and the Zakharov-Kuznetsov equations the
zero energy Novikov-Veselov equation
\begin{equation}\tag{NV}\label{NV}
  	\partial_tu+(\partial ^3 + \overline{\partial}^3)u +3(\partial (u\overline{\partial}^{-1}\partial u)+\overline{\partial}(u\partial^{-1}\overline{\partial}u))=0
\end{equation}
is another two-dimensional generalization of the famous Korteweg-de Vries
equation~(KdV). Here
\[
	\partial = \frac{\partial}{\partial z}= \frac{1}{2}(\frac{\partial}{\partial x}-i\frac{\partial}{\partial y})
	\qquad \mbox{and} \qquad
	\overline{\partial} = \frac{\partial}{\partial \overline{z}}= \frac{1}{2}(\frac{\partial}{\partial x}+i\frac{\partial}{\partial y})
\]
denote the Wirtinger derivatives. (NV) was introduced in 1984/86 by S. P. Novikov
and A. P. Veselov~\cites{NV84A, NV84B, NV86} in their study of the two-dimensional
Schrödinger operator. These authors considered the unknown function $u: \T^2
\times I \to \R$ to be a \emph{periodic} and \emph{real valued} potential.
Originally the equation was written down in the form
\begin{equation}\label{NVorg}
  	\partial_tu = (\partial ^3 + \overline{\partial}^3)u + \partial (u w) + \overline{\partial} (u \overline{w}), \qquad \overline{\partial}w=3\partial u,
\end{equation}
see equation (14) in~\cite{NV86}, which gives, if $\partial$ and
$\overline{\partial}$ can be inverted in a well-defined way, the --- for complex
valued functions differing slightly from~\eqref{NV} --- following equation
\[
	\partial_tu=(\partial ^3 + \overline{\partial}^3)u +3(\partial (u\overline{\partial}^{-1}\partial u)+\overline{\partial}(u\partial^{-1}\overline{\partial}\overline{u})).
\]
After time reversion this coincides with~\eqref{NV}, if $u$ is real. The
investigation of (NV) in the nonperiodic case via the inverse scattering method
was initiated by Boiti, Leon, Manna, and Pempinelli~\cites{BLMP86,BLMP87} and
continued later on by Tsai~\cite{T93}. Here the authors consider the potential
$u:\R^2 \times I \to \C$ to be a small, rapidly decreasing and, in general,
\emph{complex valued} function. The latter assumption is also made by
Bogdanov~\cite{Bog}, who changed the equation to
\[
	\partial_tu + (\partial ^3 + \overline{\partial}^3)u + \partial (u w_1) + \overline{\partial} (u w_2)=0, \qquad \overline{\partial}w_1=3\partial u,\partial w_2 = 3 \overline{\partial}u,
\]
to which our form~\eqref{NV} corresponds. (It turns out in our analysis that~\eqref{NV} with a $u$ as the last factor in the last term instead of the $\overline{u}$ is by far better behaved.) Bogdanov found the related equation
\[
	\partial_tv+(\partial ^3 + \overline{\partial}^3)v +3(\partial (v\overline{\partial}^{-1}\partial |v|^2)+\overline{\partial}(v\partial^{-1}\overline{\partial}|v|^2)+v\partial^{-1}\overline{\partial}(\overline{v}\overline{\partial}v)+v\overline{\partial}^{-1}\partial(\overline{v}\partial v))=0,
\]
which he called the ``modified VN equation'' (mNV), since the
Miura-type transformation
$$
\mathcal{M}: v \mapsto \mathcal{M}(v):|v|^2-i\partial v
$$
maps a solution $v$ of (mNV) with $\partial v = \overline{\partial v}$ onto a
solution $u:= \mathcal{M}(v)$ of (NV). This discovery led Bogdanov to the
conclusion that ``from the mathematical point of view [\,\ldots] the VN equation is
the natural two dimensional generalization of the KdV equation.''~\cite{Bog}*{p.~219}. (NV) is said to be completely integrable by the inverse scattering
method. The precise meaning of this statement is the subject of a lively
discussion, see e.g.~\cites{CMS, CMMPSS, LMSSI, LMSSII, Perry, MP}. As for (KdV),
smooth and --- in case of $\R^2$ being their domain --- rapidly decreasing
solutions of (NV) satisfy a whole sequence of conservation laws: Integration of
the equation over $\R^2$ or $\T^2$  gives that
$$\int u(x,y,t)dxdy=\mbox{const. },$$
which is referred to as the conservation of the mean and plays a role in our
considerations concerning the periodic case. At the level of $L^2$ we have for
solutions of~\eqref{NV} that
$$\int u(x,y,t)\overline{\partial}^{-1}\partial u(x,y,t)dxdy=\mbox{const.}$$
Unfortunately, this functional is not definite and does not give any a priori
bound for the $L^2$-norm. A recursion formula for the higher order conservation
laws is provided in~\cite{CMMPSS}*{Section~2.3}. Among them there is the
``energy''
$$E(u(t))= \int \partial u(x,y,t)\partial\overline{\partial}^{-1}\partial u(x,y,t)+ \mbox{lower order terms }dxdy= \mbox{const. },$$
which is not definite, either. It turns out that in the whole sequence of
conserved quantities there is none giving a useful a priori bound on any
$H^s$-norm. In fact, such a bound in combination with the existing local
well-posedness theory (see below) would lead to a general global well-posedness
result, eventually at a high level of regularity. But this is impossible as
illustrated by an instructive example of Taimanov and Tsarev (see~\cite{TT}*{Theorem~4}). They found a rational solution of (NV) defined on the whole plane,
decaying at infinity as $|(x,y)|^{-3}$ and developing a singularity in finite
time. As long as it exists, this solution (at fixed time $t \ge 0$) belongs to
$\bigcap_{s \ge 0}H^s(\R^2)$.

On $\R^2$ the Novikov-Veselov equation is invariant under the scaling
transformation $u \mapsto u_{\lambda}$ where, for $\lambda >0$,
$$u_{\lambda}(x,y,t)=\lambda^2u(\lambda x, \lambda y, \lambda^3t).$$
Let $u_{0,\lambda}(x,y)=u_{\lambda}(x,y,0)$. Then
$\|u_{0,\lambda}\|_{\dot{H}^{-1}}$ is independent of $\lambda$, and thus
$s_c=-1$ becomes the critical Sobolev regularity, below which we do not expect
any well-posedness result for the Cauchy problem. In fact, $C^2$-ill-posedness
in $\dot{H}^s(\R^2)$ for $s<-1$ has been shown by Angelopoulos in~\cite{Angel}*{Theorem~17}. The question of well-posedness of the Cauchy problem for (NV) has
been tackled so far with two different approaches. The first is the inverse
scattering method, which has the great advantage of leading to some global
existence theorems and to a solution formula. To the best of our knowledge, the
most advanced results in this direction are those of Perry~\cite{Perry}*{Theorem~1.6} and of Music and Perry~\cite{MP}*{Theorem 1.2}, who built on earlier
works~\cites{M, GM} of Music and of Grinevich and Manakov. The data are assumed to
belong to some weighted Sobolev space of fairly high regularity and to lie in
the image of the Miura map, or to satisfy a certain (sub-)criticality condition,
see Definition 1.1 in~\cite{MP}. Unfortunately, uniqueness and hence continuous
dependence remains open in this approach. On the other hand the Fourier
restriction norm method introduced by Bourgain in~\cites{B1, B2} and further
developed in~\cites{KPV96a, KPV96b, GTV} has been applied to treat the Cauchy
problem (nonperiodic case) for (NV) and (mNV): In~\cite{Angel} Angelopoulos
proved the local well-posedness for (NV) with data in $H^s(\R^2)$, provided that
$s>\frac{1}{2}$, and for (mNV) with data in $H^s(\R^2)$, $s>1$. His result on
(mNV) was substantially improved upon by Schottdorf in~\cite{Schott}, who could
admit $s \ge 0$ and obtain a global result for small data in the critical case
$s=0$. To treat the endpoint case he used the $U^p$- and $V^p$-spaces introduced
by Koch and Tataru~\cites{KT05, KT07, HHK}. In~\cites{KMI, KMII} Kazeykina and
Muñoz generalized the $s>\frac{1}{2}$ result mentioned above to the more general
``nonzero energy NV equation''
$$
\partial_tu+(\partial ^3 + \overline{\partial}^3)u +3(\partial (u\overline{\partial}^{-1}\partial u)+\overline{\partial}(u\partial^{-1}\overline{\partial}u)) + E(\overline{\partial}^{-1}\partial^2 u + \partial^{-1}\overline{\partial}^2u) =0,
$$
for a fixed parameter $E \in \mathbb{R}$, which is much harder to analyze. All these LWP results rely exclusively on a
global smoothing effect of solutions to the linear part of the equation,
expressed in terms of (eventually bilinear) Strichartz-type estimates with
derivative gain. Such a smoothing effect does not exist in the periodic case.

In the sequel we will follow this second approach. Additionally we will take the
structure of the nonlinearity into account, which will allow us to push down the
lower bound on $s$ in the nonperiodic case substantially and to reach something
below $L^2(\T^2)$ for data of mean zero in the periodic case. We emphasize, that our arguments do not cover the case of nonzero energy, see also the open question (3) in the last section.\\

{\bf{Acknowledgement.}} The authors wish to thank Karin Halupczok for valuable hints concerning the number theoretic aspects of the periodic case. They also want to thank the anonymous referees for valuable hints. 

\section{General arguments and main results}

We consider the initial value problem $u(x,y,0)=u_0(x,y)$ for
equation~\eqref{NV}, where either

\begin{itemize}
  \item the data $u_0$ and the solution $u(t)$ at time $t$ belong to some
    classical Sobolev space $H^s(\R^2)$ of functions defined on the whole plane
    (Cauchy problem, nonperiodic case), or
  \item $u_0$ and $u(t)$ are elements of $H^s_0(\T^2)$, the Sobolev space of (in
    both directions) periodic functions on $\R^2$ of mean zero, i.e.~we assume
    $$\int_{\T^2}u_0(x,y)dxdy=0,$$
    which is preserved under the evolution of (NV).
\end{itemize}
In contrast to the majority of the more recent literature we follow Bogdanov and
consider data and solution to be complex valued. In the end the uniqueness part
of our results will give that solutions with real valued data remain real
valued. In both cases considered here the operators
$\overline{\partial}^{-1}\partial$ and $\partial^{-1}\overline{\partial}$ are
well-defined as bounded Fourier multipliers from $H^s$ to $H^s$. To be more
explicit, let us write the Fourier transform in the space variables as
$$\F_{xy}f(\xi, \eta)= c \int e^{-ix\xi-iy\eta}f(x,y)dxdy,$$
where the integral is taken over $\R^2$ or over $\T^2$, respectively. Then we
have
$$\overline{\partial}^{-1}\partial=\F_{xy}^{-1}\frac{i\xi+\eta}{i \xi -\eta}\F_{xy}=\F_{xy}^{-1}\frac{\xi^2-\eta^2-2i\xi\eta}{\xi^2 +\eta^2}\F_{xy}=:\frac{\partial_x^2-\partial_y^2}{\Delta}-i\frac{2\partial_x\partial_y}{\Delta}$$
and
$$\partial^{-1}\overline{\partial}=\F_{xy}^{-1}\frac{i\xi-\eta}{i \xi +\eta}\F_{xy}=\F_{xy}^{-1}\frac{\xi^2-\eta^2+2i\xi\eta}{\xi^2 +\eta^2}\F_{xy}=:\frac{\partial_x^2-\partial_y^2}{\Delta}+i\frac{2\partial_x\partial_y}{\Delta}.$$
Since $\partial^3 +
\overline{\partial}^3=\frac{1}{4}(\partial_x^3-3\partial_x\partial_y^2)$ we can
rewrite equation~\eqref{NV} in real cartesian coordinates as
$$\partial_tu + \frac{1}{4}(\partial_x^3-3\partial_x\partial_y^2)u +3N(u)=0, $$
where
\begin{equation}\label{N}
  N(u)=\partial_x(u\frac{\partial_x^2-\partial_y^2}{\Delta}u)-\partial_y(u\frac{2\partial_x\partial_y}{\Delta}u).
\end{equation}
Since constant factors in front of the nonlinearity don't play any role in the
local analysis ahead, we may, after rescaling the time variable, consider the
equation
\begin{equation}\label{NVreal}
 \partial_tu + (\partial_x^3-3\partial_x\partial_y^2)u = N(u)
\end{equation}
with $N(u)$ as in~\eqref{N} and initial condition
\begin{equation}\label{IC}
  u(x,y,0)=u_0(x,y).
\end{equation}
Solutions of the linear part of this equation with initial value $u_0$ will be
denoted by $U_{\varphi}(t)u_0=e^{-it\varphi(D)}u_0$ with the phase function
$\varphi(\xi, \eta)=\xi^3-3\xi\eta^2$, which determines the weight in the
Bourgain spaces adequate for our problem. For the \emph{nonperiodic} case we
define
$$X_{s,b}:=\{f \in \Sc '(\R^3): \|f\|_{X_{s,b}}< \infty\}$$
with
\begin{align*}
\|f\|^2_{X_{s,b}}&:=\|\langle \tau - \varphi(\xi,\eta)\rangle^b\langle (\xi,\eta) \rangle^s\widehat{f}\|^2_{L^2_{\tau \xi \eta}}
\\&=\int_{\R^3}\langle \tau - \varphi(\xi,\eta)\rangle^{2b}\langle (\xi,\eta) \rangle^{2s}|\widehat{f}(\xi,\eta,\tau)|^2 d\tau d\xi d\eta ,
\end{align*}
where, for $x \in \R^n$, $\langle x \rangle = (1+|x|^2)^{\frac12}$ and
$\widehat{f}$ denotes the Fourier transform with respect to all variables
including time. The corresponding time restriction norm is denoted by
$$\|f\|_{X_{s,b}^{\delta}}:=\inf\{\|\widetilde{f}\|_{X_{s,b}}:\widetilde{f} \in X_{s,b}, \widetilde{f}\big|_{\R^2 \times (- \delta, \delta)}=f \},$$
defining our solution space, which is embedded continuously in
$C([-\delta,\delta], H^s(\R^2))$, if $b > \frac12$. Similarly, for the
\emph{periodic} case we set
$$\dot{X}_{s,b}:=\{f \in \Sc '(\R^3): f \mbox{ is periodic in space and } \|f\|_{\dot{X}_{s,b}}< \infty\},$$
where now
$$\|f\|^2_{\dot{X}_{s,b}}:= \int_{\R} \sum_{(\xi, \eta) \in \Z^2\setminus \{(0,0)\}}\langle \tau - \varphi(\xi,\eta)\rangle^{2b}\langle (\xi,\eta) \rangle^{2s}|\widehat{f}(\xi,\eta,\tau)|^2 d\tau.$$
The restriction norm spaces here are denoted by $\dot{X}^{\delta}_{s,b}$. We
will have to choose the parameter $b=\frac12$, which would lose us the embedding into a space of continous functions. In order to recover the continuity of the solution in the periodic case we will also prove estimates in the function spaces defined by
$$\|f\|^2_{\dot{Y}^s}:=\sum_{(\xi, \eta) \in \Z^2\setminus \{(0,0)\}}\left(\int_{\R}\langle \tau - \varphi(\xi,\eta)\rangle^{-1}\langle (\xi,\eta) \rangle^{s}|\widehat{f}(\xi,\eta,\tau)| d\tau \right)^2,$$
similar to those introduced in~\cite{GTV}. Now we are able to give a precise
statement of our results. Concerning the nonperiodic case we have:

\begin{theorem}\label{NVcont}
  Let $s>-\frac34$ and $u_0 \in H^s(\R^2)$. Then there exist $b > \frac12$ and
  $\delta=\delta (\|u_0\|_{H^s})>0$, such that there is a unique solution $u \in
  X_{s,b}^{\delta}$ of~\eqref{NVreal},~\eqref{IC}. Moreover, for every $R>0$ the
  solution operator
  $$S_R: H^s(\R^2) \supset B_R(0) \to X_{s,b}^{\delta(R)}, \quad u_0 \mapsto S_R(u_0):=u$$
  is Lipschitz continuous.
\end{theorem}
Similarly, for the periodic case we will prove:
\begin{theorem}\label{NVper}
  Let $s>-\frac15$ and $u_0 \in H_0^s(\T^2)$. Then there exist $\delta=\delta
  (\|u_0\|_{H^s})>0$ and a unique solution $u \in
  \dot{X}_{s,\frac12}^{\delta}\cap C([-\delta,\delta], H_0^s(\T^2))$
  of~\eqref{NVreal},~\eqref{IC}. For every $R>0$ the solution operator
  $$S_R: H_0^s(\T^2) \supset B_R(0) \to \dot{X}_{s,\frac12}^{\delta(R)}, \quad u_0 \mapsto S_R(u_0):=u$$
  is Lipschitz continuous.
\end{theorem}

\section{Symmetrization and the resonance function}

We write the nonlinearity~\eqref{N} as $N(u)=\frac12 B(u,u)$ with the bilinear operator
$$B(u,v)=\partial_x \left( \Big(\frac{\partial_x^2-\partial_y^2}{\Delta}u\Big)v + u\Big(\frac{\partial_x^2-\partial_y^2}{\Delta}v\Big)\right) - \partial_y \left( \Big(\frac{2\partial_x\partial_y}{\Delta}u\Big)v + u\Big(\frac{2\partial_x\partial_y}{\Delta}v\Big)\right).$$
Then the partial Fourier transform of $B(u,v)$ with respect to the space
variables becomes (ignoring constants and the time dependence)
\begin{align*}
  \F_{xy}B(u,v)(\xi,\eta)= \xi \int_* \left(\frac{\xi_1^2-\eta_1^2}{\xi_1^2+\eta_1^2}+\frac{\xi_2^2-\eta_2^2}{\xi_2^2+\eta_2^2} \right)\F_{xy}u(\xi_1,\eta_1)\F_{xy}v(\xi_2,\eta_2)d\xi_1d\eta_1 \\
  - \eta \int_* \left(\frac{2\xi_1\eta_1}{\xi_1^2+\eta_1^2}+\frac{2\xi_2\eta_2}{\xi_2^2+\eta_2^2} \right)\F_{xy}u(\xi_1,\eta_1)\F_{xy}v(\xi_2,\eta_2)d\xi_1d\eta_1,
\end{align*}
where $\displaystyle \int_*$ denotes integration under the convolution
constraint $(\xi,\eta)=(\xi_1,\eta_1)+(\xi_2,\eta_2)$. For the complete
multiplier in this expression an elementary calculation shows that
\begin{align*}
  m(\xi_1,\xi_2,\eta_1,\eta_2):= \xi \left(\frac{\xi_1^2-\eta_1^2}{\xi_1^2+\eta_1^2}+\frac{\xi_2^2-\eta_2^2}{\xi_2^2+\eta_2^2} \right) -\eta \left(\frac{2\xi_1\eta_1}{\xi_1^2+\eta_1^2}+\frac{2\xi_2\eta_2}{\xi_2^2+\eta_2^2} \right) \\
  = \frac{2(\xi_1\xi_2+\eta_1\eta_2)}{(\xi_1^2+\eta_1^2)(\xi_2^2+\eta_2^2)}\big(\xi(\xi_1\xi_2-\eta_1\eta_2)-\eta(\xi_1\eta_2+\xi_2\eta_1) \big).
\end{align*}
We wish to show estimates of the type
$$\|B(u,v)\|_{X_{s,b'}} \lesssim \|u\|_{X_{s,b}}\|v\|_{X_{s,b}}$$
with $s$ as low as possible and $b'=-\frac12+2\varepsilon$,
$b=\frac12+\varepsilon$, where $\varepsilon >0$ ($\varepsilon=0$ in the periodic
case). Choosing $f$, $g$ such that $\|f\|_{L^2_{\xi\eta\tau}}=\|u\|_{X_{s,b}}$ and $\|g\|_{L^2_{\xi\eta\tau}}=\|v\|_{X_{s,b}}$ the previous inequality turns into
$$\|\langle \tau-\varphi(\xi,\eta) \rangle^{b'} \langle(\xi,\eta) \rangle^s I_{f,g}\|_{L^2_{\xi\eta\tau}} \lesssim \|f\|_{L^2_{\xi\eta\tau}}\|g\|_{L^2_{\xi\eta\tau}}$$
with
\begin{align*}
	I_{f,g}(\xi,\eta,\tau):= \int_*m(\xi_1,\xi_2,\eta_1,\eta_2)\langle(\xi_1,\eta_1) \rangle^{-s}\langle \tau_1-\varphi(\xi_1,\eta_1) \rangle^{-b}f(\xi_1,\eta_1,\tau_1) \times \ldots \\ \ldots \langle(\xi_2,\eta_2) \rangle^{-s}\langle \tau_2-\varphi(\xi_2,\eta_2) \rangle^{-b}g(\xi_2,\eta_2,\tau_2)d\xi_1 d\eta_1 d \tau_1,
\end{align*}
where now $(\xi,\eta,\tau)=(\xi_1,\eta_1,\tau_1)+(\xi_2,\eta_2,\tau_2)$. $f$ and $g$ are assumed to be nonnegative and $\displaystyle \int_* \ldots d\xi_1 d\eta_1$ may denote integration with respect to the Lebesgue measure on $\R^2$ as well as alternatively the counting measure on $\Z^2 \setminus \{(0,0)\}$.
Now the resonance function,
i.e.~the quantity controlled by
$$\max\{|\tau-\varphi(\xi,\eta)|, |\tau_1-\varphi(\xi_1,\eta_1)|, |\tau_2-\varphi(\xi_2,\eta_2)|\},$$
for our nonlinearity, is given by
$$r(\xi_1,\xi_2,\eta_1,\eta_2):=\varphi(\xi,\eta)-\varphi(\xi_1,\eta_1)-\varphi(\xi_2,\eta_2)=3\big(\xi(\xi_1\xi_2-\eta_1\eta_2)-\eta(\xi_1\eta_2+\xi_2\eta_1)\big).$$
Again we leave the elementary verification of the last equality to the reader.
Comparing the expressions for $m$ and $r$ we arrive at
$$m(\xi_1,\xi_2,\eta_1,\eta_2)= \frac23 \frac{\xi_1\xi_2+\eta_1\eta_2}{(\xi_1^2+\eta_1^2)(\xi_2^2+\eta_2^2)} r(\xi_1,\xi_2,\eta_1,\eta_2),$$
which gives, for $\theta \in (0,1)$, the inequality
\begin{align}\label{resonance}
|m(\xi_1,\xi_2,\eta_1,\eta_2)|\le \frac{|r(\xi_1,\xi_2,\eta_1,\eta_2)|^{\theta}|r(\xi_1,\xi_2,\eta_1,\eta_2)|^{1-\theta}}{|(\xi_1,\eta_1)||(\xi_2,\eta_2)|}  \\ \nonumber \le |(\xi,\eta)|^{1-\theta}|(\xi_1,\eta_1)|^{-\theta}|(\xi_2,\eta_2)|^{-\theta}|r(\xi_1,\xi_2,\eta_1,\eta_2)|^{\theta},
\end{align}
the latter since $|r(\xi_1,\xi_2,\eta_1,\eta_2)|\le
|(\xi,\eta)||(\xi_1,\eta_1)||(\xi_2,\eta_2)|$. This will be used especially with
$\theta = -b' \approx \frac12$.

\section{The nonperiodic case}

In addition to the structure of the nonlinearity discussed above we will make
use of smoothing estimates of Strichartz-type for the unitary group
$(U_{\varphi}(t))_{t \in \R}$. Here and below
$I^{\sigma}=\F_{xy}^{-1}|(\xi,\eta)|^{\sigma}\F_{xy}$ represents the Riesz
potential operator of order $- \sigma$ with respect to the space variables.

\begin{lemma}
  For $u_0 \in L^2(\R^2)$ let $U_{\varphi}u_0$ denote the solution of
  $$\partial_t u+ (\partial_x^3-3\partial_x\partial_y^2)u=0 \quad \mbox{with} \quad u(0)=u_0.$$
  Then the following estimates hold true:
  \begin{itemize}
    \item If $p>3$ and $\displaystyle \frac{3}{p}+\frac{2}{q}=1$:
     \begin{equation}\label{Str0}
       \|U_{\varphi}u_0\|_{L_t^pL^q_{xy}} \lesssim \|u_0\|_{L^2_{xy}},
     \end{equation}
    \item if $p>2$ and $\displaystyle \frac{2}{p}+\frac{2}{q}=1$:
     \begin{equation}\label{Str}
       \|I^{\frac{1}{p}}U_{\varphi}u_0\|_{L_t^pL^q_{xy}} \lesssim \|u_0\|_{L^2_{xy}}.
     \end{equation}
  \end{itemize}
\end{lemma}

\emph{Citation and proof:}~\eqref{Str0} follows from~\eqref{Str} by a Sobolev
embedding. To prove~\eqref{Str} one starts with the estimate
$$\|IU_{\varphi}(t)u_0\|_{L^{\infty}_{xy}} \lesssim |t|^{-1}\|u_0\|_{L^1_{xy}},$$
which is Part 2. of Theorem 5.6 in~\cite{BKS}. For a dyadic piece of the data $P_{\Delta l} u =
\F_{xy}^{-1}\chi_{\{|(\xi,\eta)| \sim 2^l\}}\F_{xy}u$ this reads
$$\|U_{\varphi}(t)P_{\Delta l}u_0\|_{L^{\infty}_{xy}} \lesssim |t|^{-1}2^{-l}\|u_0\|_{L^1_{xy}}.$$
Now the standard proof of the Strichartz estimates using Riesz-Thorin
interpolation, the Hardy-Littlewood-Sobolev inequality and the $TT^*$-argument
applies. Since one has to deal with a gain of derivatives we refer
to~\cite{GV}*{Section~3} for more details. \hfill $\Box$

\quad \\
We remark that the endpoints $p=3$ in~\eqref{Str0} and $p=2$ in~\eqref{Str} are
excluded. Considering the results of Montgomery-Smith \cite{MS} and Tao \cite{T} we strongly believe the latter endpoint estimate to fail. By the transfer principle~\cite{GTV}*{Lemma~2.3} we obtain
corresponding $X_{s,b}$-estimates. A soft argument dealing with low frequencies
allows us to infer that
\begin{equation}\label{StrX}
  \|u\|_{L_t^pL^q_{xy}} \lesssim \|u\|_{X_{-\frac{1}{p}, b}}
\end{equation}
if $p>2$, $\displaystyle \frac{2}{p}+\frac{2}{q}=1$, and $\displaystyle b>
\frac{1}{2}$. Now we are prepared to prove the central bilinear estimate of this
section, which (inserted into the framework of Bourgain's method) leads to
Theorem 1.

\begin{prop}
  Let $ s>-\frac34$ and $ b' \le - \frac38$ as well as $ b'<s+\frac14$. Then for
  all $ b> \frac{1}{2}$ the estimate
  \begin{equation}\label{keycont}
    \|B(u,v)\|_{X_{s,b'}} \lesssim \|u\|_{X_{s,b}}\|v\|_{X_{s,b}}
  \end{equation}
  holds true.
\end{prop}

\begin{proof}
  Without loss of generality we assume $s \le - \frac58$ so that $s\le -1-b'$.
  The proof consists of a case by case discussion, essentially depending on
  which of the weights
  \begin{equation}\label{modulations}
 	\lb \tau-\varphi(\xi,\eta) \rangle, \quad \langle \tau_1-\varphi(\xi_1,\eta_1) \rangle, \quad\langle \tau_2-\varphi(\xi_2,\eta_2) \rb
  \end{equation}
  is the largest and thus controls the resonance function. We start with a
  trivial low frequency issue.\\
  \quad \\
  \emph{Case 0:} $|(\xi_1,\eta_1)|\le 1$ and $|(\xi_2,\eta_2)|\le 1$. In this
  case the multiplier $m$ is bounded, so that the left hand side
  of~\eqref{keycont} can be estimated
  \[
  \|uv\|_{L^2_{xyt}}\le \|u\|_{L^4_{xyt}}\|v\|_{L^4_{xyt}}\lesssim \|u\|_{X_{0b}}\|v\|_{X_{0b}}\lesssim \|u\|_{X_{s,b}}\|v\|_{X_{s,b}},
  \]
  where we have used~\eqref{StrX} and the support restriction of $\widehat{u}$ and
  $\widehat{v}$ to $\{|(\xi,\eta)| \le 1\}$.\\
  \quad \\
  \emph{Case 1:} $\langle \tau-\varphi(\xi,\eta) \rangle$ is maximal.\\
  \emph{Subcase 1.1:} $|(\xi_1,\eta_1)|\le 1 \le |(\xi_2,\eta_2)|$. In this case
  we have $|(\xi,\eta)| \sim |(\xi_2,\eta_2)|$ which reduces the consideration
  to the case $s=0$. We use~\eqref{resonance} with $\theta = \frac38$ to obtain
  $$\|B(u,v)\|_{X_{0b'}} \lesssim \|I^{-\theta}u\|_{L^4_{xyt}}\|I^{\frac14}v\|_{L^4_{xyt}},$$
  where by~\eqref{StrX} the second factor is bounded by $\|v\|_{X_{0b}}$. For
  the first factor we use a Sobolev embedding and the fact that $\widehat{u}$ is
  restricted to $\{|(\xi,\eta)| \le 1\}$ to see that
  $$\|I^{-\theta}u\|_{L^4_{xyt}} \lesssim \|u\|_{L_t^4L^2_{xy}} \lesssim \|u\|_{X_{0b}},$$
  where in the last step a time embedding was applied.\\
  \emph{Subcase 1.2:} $|(\xi_2,\eta_2)|\le 1 \le |(\xi_1,\eta_1)|$ needs no
  discussion by symmetry.\\
  \emph{Subcase 1.3:} $|(\xi_1,\eta_1)|\ge 1$ and $|(\xi_2,\eta_2)| \ge 1$. We
  use~\eqref{resonance} with $\theta = -b'$ and without loss of generality $s+1+b'\le0$ to infer
  that the contribution of this case is bounded by
  $$\|(I^{b'}u)(I^{b'}v)\|_{L^2_{xyt}}\le \|I^{b'}u\|_{L^4_{xyt}}\|I^{b'}v\|_{L^4_{xyt}}\lesssim \|u\|_{X_{s,b}}\|v\|_{X_{s,b}},$$
  the latter by~\eqref{StrX} and the assumption $b'-\frac14 < s$.

  \quad \\
  \emph{Case 2:} $\langle \tau_1-\varphi(\xi_1,\eta_1) \rangle$ is maximal. \\
  \emph{Subcase 2.1:} $|(\xi_1,\eta_1)|\le 1 \le |(\xi_2,\eta_2)|$. Because of
  $|(\xi,\eta)| \sim |(\xi_2,\eta_2)|$ we may consider $s=0$ only. We write
  $\Lambda^b=\F^{-1}\langle \tau-\varphi(\xi,\eta) \rangle^b \F$ and
  use~\eqref{resonance} with $\theta = \frac38$ to see that the contribution of
  this region is bounded by
  \begin{align*}
    \|I^{\frac14}((I^{-\frac{3}{8}}\Lambda^{\frac{3}{8}+b+b'}u)I^{1-\frac34-\frac14}v)\|_{X_{0,-b}}
    \lesssim \|(I^{-\frac{3}{8}}\Lambda^{\frac{3}{8}+b+b'}u)v\|_{L^{\frac43}_{xyt}} \\
    \lesssim \|I^{-\frac{3}{8}}\Lambda^{\frac{3}{8}+b+b'}u\|_{L^2_tL^4_{xy}}\|v\|_{L^4_tL^2_{xy}}
    \lesssim \|\Lambda^bu\|_{L^2_{xyt}}\|v\|_{X_{0,\frac14}}\le\|u\|_{X_{0,b}}\|v\|_{X_{0,b}}.
  \end{align*}
  Here we have used the dual version of the $L^4$-Strichartz-type estimate,
  H\"older's inequality and Sobolev-type embeddings in space (first factor) and
  time (second factor).\\
  \emph{Subcase 2.2:} $|(\xi_2,\eta_2)|\le 1 \le |(\xi_1,\eta_1)|$. Considering
  again $s=0$ and choosing $\theta = \frac38$ in~\eqref{resonance} we get the
  bound
  $$\|\Lambda^{\frac{3}{8}+b+b'}u\|_{L^2_{xyt}}\|I^{-\frac38}v\|_{L^4_{xyt}} \lesssim \|u\|_{X_{0,b}}\|v\|_{X_{0,b}}.$$
  \emph{Subcase 2.3:} $|(\xi_1,\eta_1)|\ge 1$ and $|(\xi_2,\eta_2)| \ge 1$. Here
  we choose $\theta = b'$ in~\eqref{resonance}, remember that $s+1+b'\le0$ and
  obtain the bound
  \begin{equation}\label{bd}
    \|(I^{b'}\Lambda^bu)(I^{b'}v)\|_{X_{0,-b}}.
  \end{equation}
  Now there are two possibilities:\\
  \emph{2.3.1:} $|(\xi_1,\eta_1)| \lesssim |(\xi,\eta)|$. We use the dual
  version of the $L^4$-Strichartz-type estimate, H\"older, and the estimate
  itself for the second factor to get
  $$\eqref{bd} \lesssim \|(I^{b'-\frac14}\Lambda^bu)(I^{b'}v)\|_{L^{\frac43_{xyt}}}\lesssim \|I^{b'-\frac14}\Lambda^{b}u\|_{L^2_{xyt}}\|I^{b'}v\|_{L^4_{xyt}}\lesssim \|u\|_{X_{s,b}}\|v\|_{X_{s,b}}.$$
  \emph{2.3.2:} $|(\xi_1,\eta_1)| \lesssim |(\xi_2,\eta_2)|$. We start with a
  time embedding, apply Hölder's inequality, a Sobolev embedding in space and
  the almost endpoint version of the Strichartz-type estimate to obtain
  \begin{align*}
    \eqref{bd} & \lesssim \|(I^{b'-\frac14+}\Lambda^bu)(I^{b'+\frac14-}v)\|_{L_t^{1+}L^2_{xy}} \\
    & \lesssim \|I^{b'-\frac14+}\Lambda^{b}u\|_{L_t^2L^{2+}_{xy}}\|I^{b'}v\|_{L_t^{2+}L^{\infty-}_{xy}}\lesssim \|u\|_{X_{s,b}}\|v\|_{X_{s,b}}.
  \end{align*}
  The third case, where $\langle \tau_2-\varphi(\xi_2,\eta_2) \rangle$ is
  maximal, needs no consideration by symmetry.
\end{proof}

\section{The periodic case}

To prove a bilinear Strichartz-type estimate for the periodic problem, we rely on the following number theoretic result due to W.~M.~Schmidt:

\begin{theorem*}[Schmidt]
	Call $n(\mathfrak{C}, N)$ the number of integral points on the curve $\mathfrak{C} = \set{(x, f(x)) \mid x \in \mathbb{R}}$ in an arbitrary square of side length $N \ge 1$. Then, if $f''$ exists and is weakly monotonic, the estimate
	\begin{equation}\label{generalNTestimate}
		n(\mathfrak{C}, N) \le c(\e) N^{\gamma + \e}
	\end{equation}
	holds true for $\gamma = \frac{3}{5}$ with a constant $c(\e)$ independent of the particular curve.
\end{theorem*}

\noindent See \cite{Schmidt}*{Theorem~1}. We will apply this estimate to
\begin{enumerate}[(i)]
	\item classical hyperbolas described by
	\[
	a(x^2 - y^2) + 2bxy = c \qquad (c \not= 0)
	\]
	and to
	\item cubic hyperbola-like curves of the form
	\[
	(x + a)(x^2 - y^2) = 2(y + b)xy,
	\]
\end{enumerate}
where $a$, $b$ and $c$ are parameters. Schmidt's Theorem applies to these curves, unless they degenerate (partially) into straight lines. It is possible that sharper estimates with lower exponents $\gamma$ hold true for the curves in (i) and (ii). Thus we decided to state and prove several subsequent estimates depending on the exponent $\gamma \in [0, 1)$, assuming~\eqref{generalNTestimate} to be applicable.

Next we define the bilinear projection operator $Q$ by
\[
\widehat{Q(u, v)}(\xi, \eta, \tau) = \sum_* (1 - \delta_{\xi,0}\delta_{\xi_1,0}) \hat{u}(\xi_1, \eta_1, \tau_1) \hat{v}(\xi_2, \eta_2, \tau_2)
\]
where $\sum_*$ indicates summation under the constraint introduced by the convolution $(\xi, \eta, \tau) = (\xi_1, \eta_1, \tau_1) + (\xi_2, \eta_2, \tau_2)$. $Q$ acts only on the first space variable.

\begin{prop}
	Let $\gamma \in [0,1)$, such that~\eqref{generalNTestimate} holds for the nondegenerate curves of type (i) and (ii). For $B_R \subset \R^2$ a circle with radius $R > 0$ and arbitrary center and $u_0, v_0 \in L^2_{xy}$, where $\supp{\hat{u}_0} \subset B_R$, one has
	\begin{equation}\label{bilinOrig}
		\|Q(U_\varphi u_0, U_\varphi v_0)\|_{L^2_{xyt}} \lesssim R^{\frac{\gamma}{2}+} \|u_0\|_{L^2_{xy}} \|v_0\|_{L^2_{xy}}.
	\end{equation}
\end{prop}
\begin{remark*}
	Without the projector $Q$ the best possible estimate is
	\[
		\|U_\varphi u_0 U_\varphi v_0\|_{L^2_{xyt}} \lesssim R^{\frac{1}{2}} \|u_0\|_{L^2_{xy}} \|v_0\|_{L^2_{xy}},
	\]
	which can be seen by the example $\hat{u}_0(\xi,\eta) = \hat{v}_0(\xi,\eta) = \delta_{\xi, 0}\chi_{[-R,R]}(\eta)$. But~\eqref{bilinOrig} will work in our application to the nonlinearity, since the bilinear Fourier multiplier $m$ introduced at the beginning of Section 3 vanishes, if $\xi = \xi_1 = \xi_2 = 0$.
\end{remark*}

\begin{proof}
	We split
	\begin{align*}
	\F_{xyt}Q(U_\varphi u_0, U_\varphi v_0)(\xi, \eta, \tau) &= \sum_* (1 - \delta_{\xi,0}\delta_{\xi_1,0}) \delta_{\tau,\varphi(\xi_1, \eta_1) + \varphi(\xi_2, \eta_2)} \hat{u}_0(\xi_1, \eta_1) \hat{v}_0(\xi_2, \eta_2) \\\nonumber &= \RN{1} + \RN{2},
	\end{align*}
	where for $\RN{1}$ we assume that $\tau - \frac{\xi^3}{4} + \frac{7}{4}\xi\eta^2 \not= 0$. This term can be estimated by Cauchy-Schwarz
	\begin{align}\label{afterFTCS}
		\|\RN{1}\|_{L^2_{\xi\eta\tau}}^2 &\lesssim \sum_{(\xi, \eta,\tau) \in \Z^3} \Sigma_1(\xi, \eta, \tau) \sum_* \delta_{\tau, \varphi(\xi_1, \eta_1) + \varphi(\xi_2, \eta_2)} |\hat{u}_0(\xi_1, \eta_1) \hat{v}_0(\xi_2, \eta_2)|^2,
	\end{align}
	noting that $\delta_{\tau, \varphi(\xi_1, \eta_1) + \varphi(\xi_2, \eta_2)}^2 = \delta_{\tau, \varphi(\xi_1, \eta_1) + \varphi(\xi_2, \eta_2)}$ and $\hat{u_0} = \chi_{R} \hat{u}_0$ where we define
	\[
		\Sigma_1(\xi, \eta, \tau) = \sum_* \delta_{\tau, \varphi(\xi_1, \eta_1) + \varphi(\xi_2, \eta_2)} \chi_R(\xi_1, \eta_1).
	\]

	If we are now able to prove an estimate of the type $\Sigma_1(\xi, \eta, \tau) \lesssim R^{\gamma+}$ we can further bound
	\begin{align*}
	\eqref{afterFTCS}&\lesssim R^{\gamma+} \sum_{(\xi, \eta) \in \Z^2} \sum_* \left(\sum_{\tau \in \Z} \delta_{\tau, \varphi(\xi_1, \eta_1) + \varphi(\xi_2, \eta_2)}\right) |\hat{u}_0(\xi_1, \eta_1) \hat{v}_0(\xi_2, \eta_2)|^2\\
	&\le R^{\gamma+} \|u_0\|_{L^2_{xy}}^2 \|v_0\|_{L^2_{xy}}^2
	\end{align*}
	which is our proposition for the contribution by $\RN{1}$. In order to bound $\Sigma_1(\xi, \eta, \tau)$ we use the substitution $\xi_1 = x + \frac{\xi}{2}$ and $\eta_1 = y + \frac{\eta}{2}$. A lengthy but elementary calculation shows that then
	\[
		\tau - \varphi(\xi_1, \eta_1) - \varphi(\xi_2, \eta_2) = \tau - \frac{1}{4}\xi^3 + \frac{7}{4} \xi \eta^2 + 3\xi(x^2 - y^2) - 6\eta xy =: K(\xi, \eta, \tau, x, y).
	\]
	One immediately identifies this to be a curve of type (i) in the variables $x$ and $y$, $(\xi, \eta)$ and $\tau$ only play the role of parameters. The sum to be estimated now reads
	\[
		\Sigma_1(\xi, \eta, \tau) = \sum_{(x,y)\in\Z^2} \delta_{0, K(\xi, \eta, \tau, x, y)} \chi_{2R}(2x + \xi, 2\eta  + y)
	\]
	where, because of the substitution, we have had to double the radius of the circle. Now the general result~\eqref{generalNTestimate} about curves is applicable, since this sum merely counts the integral points within some disc of radius $\lesssim R$ on the hyperbola $K$. Hence, as desired, $\Sigma_1(\xi, \eta, \tau) \lesssim R^{\gamma+}$ and this completes the proof for $\RN{1}$.
	
	The second contribution is, with $(x, y) = (\xi - 2\xi_1, \eta - 2\eta_1)$,
	\begin{align*}
		\RN{2} &= \delta_{\tau, \frac{\xi^3}{4} + \frac{7}{4}\xi\eta^2} \sum_* \delta_{(x + 2\xi_1)(x^2 - y^2), 2(y + 2\eta_1) xy} (1 - \delta_{\xi,0}\delta_{\xi_1,0}) \hat{u}_0(\xi_1, \eta_1) \hat{v}_0(\xi_2, \eta_2)\\
		& =: \delta_{\tau, \frac{\xi^3}{4} + \frac{7}{4}\xi\eta^2} \cdot \Sigma_2(\xi, \eta).
	\end{align*}
	To estimate $\|\RN{2}\|_{L^2_{\xi\eta\tau}} = \|\Sigma_2\|_{L^2_{\xi\eta}}$ we decompose $\mathbb{R}^2 = \sum_{\alpha \in \mathbb{Z}^2} Q_\alpha$, where $Q_\alpha$ are disjoint squares of side length $2R$, so that
	\begin{equation}
		\|\RN{2}\|_{L^2_{\xi\eta\tau}}^2 = \sum_{\alpha \in \mathbb{Z}^2} \|\chi_{Q_\alpha} \Sigma_2\|_{L^2_{\xi\eta}}^2.
	\end{equation}
	Next we estimate $\|\chi_{Q_\alpha} \Sigma_2\|_{L^2_{\xi\eta}}$ for $\alpha \in \mathbb{Z}^2$ fixed by duality. For that purpose let $\psi \in L^2_{\xi\eta}$ with $\|\psi\|_{L^2_{\xi\eta}} \le 1$. Then
	\begin{align}\label{dualityEst}
		\langle \psi, \chi_{Q_\alpha} \Sigma_2 \rangle_{L^2_{\xi\eta}} &= \sum_{(\xi,\eta) \in \Z^2} \psi(\xi,\eta) \chi_{Q_\alpha}(\xi,\eta) \sum_* \hat{u}_0(\xi_1, \eta_1) \hat{v}_0(\xi_2, \eta_2) \times \ldots \\ \nonumber &\qquad\ldots  (1 - \delta_{\xi,0}\delta_{\xi_1,0}) \delta_{(x + 2\xi_1)(x^2 - y^2), 2(y + 2\eta_1) xy} \\
		\nonumber &= \sum_{(\xi_1, \eta_1) \in \Z^2} \hat{u}_0(\xi_1, \eta_1) \sum_{(\xi,\eta) \in \mathbb{Z}^2} \chi_{Q_\alpha}(\xi,\eta) \psi(\xi,\eta) \hat{v}_0(\xi_2, \eta_2) \times\ldots \\ \nonumber & \qquad\ldots (1 - \delta_{\xi,0}\delta_{\xi_1,0}) \delta_{(x + 2\xi_1)(x^2 - y^2), 2(y + 2\eta_1) xy}.
	\end{align}
	An application of Cauchy-Schwarz' inequality to the inner sum gives
	\begin{align}\label{squareContrib}
		& \sum_{(\xi,\eta) \in \mathbb{Z}^2} \chi_{Q_\alpha}(\xi,\eta) \psi(\xi,\eta) (1 - \delta_{\xi,0}\delta_{\xi_1,0}) \delta_{(x + 2\xi_1)(x^2 - y^2), 2(y + 2\eta_1) xy} \hat{v}_0(\xi_2, \eta_2) \\ \nonumber & \qquad \le \Sigma_3(\xi_1, \eta_1)^\frac{1}{2} \cdot \left( \sum_{(\xi,\eta) \in \mathbb{Z}^2} |\psi(\xi,\eta)|^2 |\hat{v}_0(\xi - \xi_1, \eta - \eta_1)|^2 \right)^\frac{1}{2}
	\end{align}
	Where we have shortened the first factor to
	\begin{equation}\label{type2count}
	\Sigma_3(\xi_1, \eta_1) := \sum_{(\xi,\eta) \in \mathbb{Z}^2} (1 - \delta_{\xi,0}\delta_{\xi_1,0}) \delta_{(x + 2\xi_1)(x^2 - y^2), 2(y + 2\eta_1) xy} \chi_{Q_\alpha}(\xi,\eta),
	\end{equation}
	the variables $(\xi_1, \eta_1)$ now appearing as parameters. Here again we must argue for an estimate of type $\Sigma_3(\xi_1, \eta_1) \lesssim R^{\gamma+}$, similar as to the above.
	In general, there are three kinds of solutions to the hyperbola-like curve of type (ii) appearing in this sum:
	\begin{enumerate}[(i)]
		\item If $\xi_1 = 0$ and $x = 0$, then an arbitrary pair $(\eta_1, y) \in \Z^2$ will complete a solution to $(x + 2\xi_1)(x^2 - y^2) = 2(y + 2\eta_1) xy$. Though since $x = \xi - 2\xi_1$ the factor involving the first Kronecker deltas causes these solutions to be disregarded in the count.
		
		\item In case $\xi_1\eta_1 \not= 0$ then $(x, y) = (-\frac{2}{3}\xi_1, -\frac{2}{9}\frac{\xi_1^2}{\eta_1})$ also gives a solution on the curve. Though since this is just a single point -- $(\xi_1, \eta_1)$ are fixed -- it may at most give a single 1 in our sum.
		
		\item Lastly, if $3x + 2\xi_1 \not= 0$ then
		\[
		y_\pm = \frac{\pm\sqrt{x^2(4\xi_1^2 + 4\eta_1^2 + 8\xi_1 x + 3x^2)} - 2\eta_1x}{3x+2\xi_1}
		\]
		gives a whole family of solutions depending on $x$. In order to ensure that Schmidt's Theorem is sufficient to give the required bound, we must ensure that if such a curve degenerates into a straight line, it has an irrational slope. Assuming $y_\pm$ does indeed describe a straight line we may calculate its slope as $\lim_{x \to \infty} \frac{y_\pm}{x} = \frac{\pm 1}{\sqrt{3}}$, which is irrational. In all other cases Schmidt's theorem delivers the required bound $\Sigma_3(\xi_1, \eta_1) \lesssim R^{\gamma+}$.
	\end{enumerate}
	
	Inserting this into~\eqref{squareContrib}, then into~\eqref{dualityEst} and applying Cauchy-Schwarz to the outer sum over $(\xi_1,\eta_1) \in \mathbb{Z}^2$ we arrive at
	\[
		\langle \psi, \chi_{Q_\alpha} \Sigma_2 \rangle_{L^2_{\xi\eta}} \lesssim R^{\frac{\gamma}{2}+} \|\hat{u}_0\|_{L^2_{\xi\eta}} \|\hat{v}_0\|_{L^2_{\xi\eta}}.
	\]
	Since in the above calculation we have $(\xi_1, \eta_1) \in B_R$ and $(\xi,\eta) \in Q_\alpha$, the variables $(\xi_2, \eta_2) = (\xi, \eta) - (\xi_1, \eta_1)$ are confined to a square $\tilde{Q}_\alpha$ of side length $4R$ containing $Q_\alpha - B_R$, so that in fact we can rely on the stronger estimates
	\[
		\langle \psi, \chi_{Q_\alpha} \Sigma_2 \rangle_{L^2_{\xi\eta}} \lesssim R^{\frac{\gamma}{2}+} \|\hat{u}_0\|_{L^2_{\xi\eta}} \|\chi_{\tilde{Q}_\alpha} \hat{v}_0\|_{L^2_{\xi\eta}}
	\]
	respectively on
	\[
		\|\chi_{Q_\alpha} \Sigma_2 \|_{L^2_{\xi\eta}}^2 \lesssim R^{\gamma+} \|\hat{u}_0\|_{L^2_{\xi\eta}}^2 \|\chi_{\tilde{Q}_\alpha} \hat{v}_0\|_{L^2_{\xi\eta}}^2
	\]
	Since the $\tilde{Q}_\alpha$ can be chosen in such a way that their union covers $\mathbb{R}^2$ exactly four times, we can sum over $\alpha \in \mathbb{Z}^2$ to obtain
	\[
		\|\RN{2}\|_{L^2_{\xi\eta\tau}}^2 \lesssim \sum_{\alpha \in \mathbb{Z}^2} R^{\gamma+} \|\hat{u}_0\|_{L^2_{\xi\eta}}^2 \|\chi_{\tilde{Q}_\alpha} \hat{v}_0\|_{L^2_{\xi\eta}}^2 \lesssim R^{\gamma+} \|\hat{u}_0\|_{L^2_{\xi\eta}}^2 \|\hat{v}_0\|_{L^2_{\xi\eta}}^2
	\]
	which by Plancherel gives the desired bound.
\end{proof}
So that we can make use of this estimate we will first use the transfer principle
\[
	\|Q(u,v)\|_{L^2_{xyt}} \lesssim \|u\|_{X_{\frac{\gamma}{2}+,b}} \|v\|_{X_{0,b}},
\]
which holds for any $b > \frac{1}{2}$, and also interpolate this with the trivial bound
\[
	\|Q(u,v)\|_{L^2_{xyt}} \le \|u\|_{L^4_t L^\infty_{xy}} \|v\|_{L^4_t L^2_{xy}} \lesssim \|u\|_{X_{1+, \frac{1}{4}}} \|v\|_{X_{0, \frac{1}{4}}},
\]
in order to arrive at
\begin{equation}\label{bilinXsb}
	\|Q(u,v)\|_{L^2_{xyt}} \lesssim \|u\|_{X_{\frac{\gamma}{2}+,\frac{1}{2}-}} \|v\|_{X_{0,\frac{1}{2}-}}.
\end{equation}
Dualizing we obtain
\begin{equation}\label{dualBilinXsb}
	\|Q(u,v)\|_{X_{0,-\frac{1}{2}+}} \lesssim \|u\|_{L^2_{xyt}} \|v\|_{X_{\frac{\gamma}{2}+,\frac{1}{2}-}}.
\end{equation}
One additional estimate is needed, which we prove with a second dyadic decomposition.
\begin{lemma}
	Assume that~\eqref{generalNTestimate} holds with a certain $\gamma \in [0,1)$ for the nondegenerate curves of type (i) and (ii). Then
	\begin{equation}\label{transposedBilin}
		\|Q(u,v)\|_{L^2_t H^{-\frac{\gamma}{2}-}} \lesssim \|u\|_{X_{0,\frac{1}{2}-}} \|v\|_{X_{0,\frac{1}{2}-}}.
	\end{equation}
\end{lemma}
\begin{proof}
	With $\widehat{Q_0(u,v)}(\xi,\eta,\tau) = \delta_{\xi, 0} \hat{u}(\xi,\eta,\tau)$ we can write
	\[
		Q(u, v) = ((I-Q_0)u)v + (Q_0 u)(I-Q_0)v,
	\]
	and the Fourier transform of both contributions vanishes, if $\xi_1 = \xi_2 = \xi = 0$, so that~\eqref{bilinXsb} applies to both of them. We give the argument only for the first, which we write as $wv$ with $w = (I - Q_0)u$. Using a dyadic decomposition in the space variables only with Littlewood-Paley projections $P_{\Delta l} = \F_{xy}^{-1}\chi_{\set{|(\xi, \eta)| \sim 2^l}}\F_{xy}$, $l \ge 1$, and $P_{\Delta 0} = \F_{xy}^{-1}\chi_{\set{|(\xi, \eta)| \le 1}}\F_{xy}$ we obtain
	\[
		\|wv\|_{L^2_t H^{-\frac{\gamma}{2}-}} \le \sum_{l \ge 0} 2^{-l(\frac{\gamma}{2} + \e)} \|P_{\Delta l}(uv)\|_{L^2_{xyt}}.
	\]
	Now for a fixed $l \in \N_0$ we write
	\begin{equation}\label{squareDecomp}
		\|P_{\Delta l}(wv)\|^2_{L^2_{xyt}} = \sum_{\alpha, \beta \in \Z^2} \lb P_{\Delta l}(P_{Q_\alpha^l}(w) \cdot v), P_{\Delta l}(P_{Q_\beta^l}(w) \cdot v) \rb
	\end{equation}
	where we have introduced a second dyadic decompostion with squares $Q_\alpha^l$ of side length $2^l$, centered at $\alpha 2^l$ with $\alpha \in \Z^2$. Double sized squares with the same centers will be denoted $\tilde{Q}_\alpha^l$.
	
	Hence if $(\xi_1, \eta_1) \in Q_\alpha^l$ and $|(\xi, \eta)| \le 2^l$ then we must have $(\xi_2, \eta_2) = (\xi, \eta) - (\xi_1, \eta_1) \in \tilde{Q}_{-\alpha}^l$, so we can estimate
	\begin{align*}
		\eqref{squareDecomp} &= \sum_{\alpha, \beta \in \Z^2} \lb P_{Q_\alpha^l}(w) \cdot P_{\tilde{Q}_{-\alpha}^l}(v), P_{Q_\beta^l}(w) \cdot P_{\tilde{Q}_{-\beta}} (v) \rb \\
		&\le \sum_{\alpha, \beta \in \Z^2} \lb P_{\tilde{Q}_\alpha^l} (w) \cdot P_{\tilde{Q}_{\beta}^l} (\overline{v}), P_{\tilde{Q}_\beta^l}(w) \cdot P_{\tilde{Q}_{\alpha}}(\overline{v}) \rb \\
		&\le \sum_{\alpha, \beta \in \Z^2} \|P_{\tilde{Q}_\alpha^l} (w) \cdot P_{\tilde{Q}_{-\beta}^l} (v)\|_{L^2_{xyt}} \|P_{\tilde{Q}_\beta^l} (w) \cdot P_{\tilde{Q}_{-\alpha}} (v)\|_{L^2_{xyt}} \\
		&\le \sum_{\alpha, \beta \in \Z^2} \|P_{\tilde{Q}_\alpha^l} (w) \cdot P_{\tilde{Q}_{-\beta}^l} (v)\|_{L^2_{xyt}}^2 \\
		&\lesssim 2^{l(\gamma + \frac{\e}{2})} \sum_{\alpha, \beta \in \Z^2} \|P_{\tilde{Q}_\alpha^l} w\|_{X_{0,\frac{1}{2}-}}^2 \| P_{\tilde{Q}_{-\beta}^l} v\|_{X_{0,\frac{1}{2}-}}^2 \\
		&\lesssim 2^{l(\gamma + \frac{\e}{2})} \|w\|_{X_{0,\frac{1}{2}-}}^2 \|v\|_{X_{0,\frac{1}{2}-}}^2 \lesssim 2^{l(\gamma + \frac{\e}{2})} \|u\|_{X_{0,\frac{1}{2}-}}^2 \|v\|_{X_{0,\frac{1}{2}-}}^2
	\end{align*}
	Here we have used Cauchy-Schwarz twice, the $X_{s,b}$-estimate~\eqref{bilinXsb} and the almost orthogonality of the sequences $(P_{\tilde{Q}_\alpha^l} w)_{\alpha \in \Z^2}$ and $(P_{\tilde{Q}_{-\beta}^l} v)_{\beta \in \Z^2}$. Altogether
	\begin{align*}
		\|wv\|_{L^2_t H^{-\frac{\gamma}{2}-}} &\lesssim \sum_{l \ge 0} 2^{-l(\frac{\gamma}{2} + \e)} 2^{l(\frac{\gamma}{2} + \frac{\e}{2})} \|u\|_{X_{0,\frac{1}{2}-}} \|v\|_{X_{0,\frac{1}{2}-}} \\
		&\lesssim \|u\|_{X_{0,\frac{1}{2}-}} \|v\|_{X_{0,\frac{1}{2}-}}
	\end{align*}
	as desired.
\end{proof}
Now we are prepared to show the proposition that, when inserted into the general framework of Bourgain's $X_{s,b}$-spaces, will result in a well-posedness theorem.
\begin{prop}\label{generalPerEstimate}
	Let $\gamma \in [0, 1)$, such that~\eqref{generalNTestimate} holds for nondegenerate curves of type (i) and (ii), and $s > \frac{\gamma - 1}{2}$, then for all $u, v \in \dot{X}_{s, \frac{1}{2}}$ with support in $\R^2 \times [-\delta, \delta]$ there exists an $\e > 0$ such that
	\begin{align*}
		&\|B(u,v)\|_{\dot{X}_{s,-\frac{1}{2}}} \lesssim \delta^\e \|u\|_{\dot{X}_{s,\frac{1}{2}}} \|v\|_{\dot{X}_{s,\frac{1}{2}}} &&\quad\text{and}\\
		&\|B(u,v)\|_{\dot{Y}^s} \lesssim \delta^\e \|u\|_{\dot{X}_{s,\frac{1}{2}}} \|v\|_{\dot{X}_{s,\frac{1}{2}}} &&\quad\text{hold.}
	\end{align*}
\end{prop}
\begin{proof}
	Since our data is of mean zero we may use $|(\xi_i, \eta_i)| \sim \lb (\xi_i, \eta_i) \rb$ for $i \in \{1,2\}$, and will do so freely without further mention. We also assume $s < 0$, because $\gamma < 1$. As in the nonperiodic case the proof is split into cases where a single one of the modulations \eqref{modulations} is maximal.
	\quad \\
	\emph{Case 1:} $\lb \tau - \varphi(\xi, \eta) \rb$ is maximal. Without loss of generality we may assume that $|(\xi_1, \eta_1)| \gtrsim |(\xi_2, \eta_2)|$. Making use of~\eqref{resonance} with $\theta = \frac{1}{2}$ we can estimate
	\begin{align}
		\|B(u, v)\|_{\dot{X}_{s, -\frac{1}{2}}} &\lesssim \|Q(I^{-\frac{1}{2}}u,I^{-\frac{1}{2}}v)\|_{\dot{X}_{\frac{1}{2}+s,0}} \lesssim \|Q(I^s u, I^{-\frac{1}{2}} v)\|_{L^2_{xyt}}\\
		&\lesssim \|u\|_{\dot{X}_{s,\frac{1}{2}-}} \|v\|_{\dot{X}_{\frac{\gamma - 1}{2}+,\frac{1}{2}-}} \lesssim \delta^\e \|u\|_{\dot{X}_{s,\frac{1}{2}}} \|v\|_{\dot{X}_{s,\frac{1}{2}}}.
	\end{align}
	In the penultimate step we used our bilinear estimate~\eqref{bilinXsb}. The last step depends on the support condition on $u$ and $v$.
	\quad \\
	\emph{Case 2:} $\lb \tau_1 - \varphi(\xi_1, \eta_1) \rb$ is maximal. Again we begin this case by using~\eqref{resonance} with $\theta = \frac{1}{2}$, though now we must use the modulation on the first factor to eliminate the resonance function:
	\begin{equation}\label{startingEst}
		\|B(u, v)\|_{\dot{X}_{s, -\frac{1}{2}}} \lesssim \|B(u, v)\|_{\dot{X}_{s, -\frac{1}{2}+}} \lesssim \|Q(I^{-\frac{1}{2}}\Lambda^{\frac{1}{2}}u, I^{-\frac{1}{2}}v)\|_{\dot{X}_{\frac{1}{2}+s,-\frac{1}{2}+}}
	\end{equation}
	The first bound may seem trivial and unnecessary, but we will come back to it in bounding the $Y^s$-norms. Depending on which factor the derivatives on the product can now fall we must differentiate between two cases:\\
	\emph{Subcase 2.1:} $|(\xi_1, \eta_1)| \gtrsim |(\xi_2, \eta_2)|$. Here the derivatives can only fall on the first factor, so we use~\eqref{dualBilinXsb} putting $u$ into $L^2_{xyt}$ and lastly using the support condition again:
	\[
		\eqref{startingEst} \lesssim \|Q(I^{s}\Lambda^{\frac{1}{2}}u, I^{-\frac{1}{2}}v)\|_{\dot{X}_{0,-\frac{1}{2}+}} \lesssim \|u\|_{\dot{X}_{s,\frac{1}{2}}} \|v\|_{\dot{X}_{\frac{\gamma - 1}{2}+,\frac{1}{2}-}} \lesssim \delta^\e \|u\|_{\dot{X}_{s,\frac{1}{2}}} \|v\|_{\dot{X}_{s,\frac{1}{2}}}.
	\]
	\emph{Subcase 2.2:} $|(\xi_1, \eta_1)| \lesssim |(\xi_2, \eta_2)|$. This time we use the dual of~\eqref{transposedBilin} putting the first factor in $L^2_t H^{\frac{\gamma}{2}+}$:
	\[
		\eqref{startingEst} \lesssim \|Q(I^{-\frac{1}{2}}\Lambda^{\frac{1}{2}}u, I^{s}v)\|_{\dot{X}_{0,-\frac{1}{2}+}} \lesssim \delta^\e \|u\|_{\dot{X}_{s,\frac{1}{2}}} \|v\|_{\dot{X}_{s,\frac{1}{2}}}
	\]
	The case where $\lb \tau_2 - \varphi(\xi_2, \eta_2) \rb$ is maximal need not be considered by symmetry.

	Next we can deal with the $Y^s$-norm estimate. Here again we consider two cases, where either the modulation of the product or of the first factor is maximal.\\
	\quad \\
	\emph{Case 1:} $\lb \tau - \varphi(\xi, \eta) \rb$ is maximal. Using $\theta = 1-$ in~\eqref{resonance} we can make nearly complete use of the modulation. Discarding the derivative gain (and remainder of the modulation) on the product and after applying Cauchy-Schwarz twice we arrive at the desired bound
	\begin{align*}
	\|B(u,v)\|_{\dot{Y}^s} &\lesssim \|I^{s+}\Lambda^{0-} Q(I^{-1+}u, I^{-1+}v)\|_{L^2_{\xi\eta}L^1_\tau} \lesssim \|Q(I^{-1+}u, I^{-1+}v)\|_{L^2_{\xi\eta}L^1_\tau}\\
	&\lesssim \|u\|_{\dot{X}_{s,\frac{1}{2}-}} \|v\|_{\dot{X}_{s,\frac{1}{2}-}} \lesssim \delta^\e \|u\|_{\dot{X}_{s,\frac{1}{2}}} \|v\|_{\dot{X}_{s,\frac{1}{2}}}.
	\end{align*}
	\emph{Case 2:} $\lb \tau_1 - \varphi(\xi_1, \eta_1) \rb$ is maximal. Here we may use just over half of the modulation to apply Cauchy-Schwarz in the $\tau$ variable. This results in the same situation as in after the first inequality in~\eqref{startingEst}. The case where $\lb \tau_2 - \varphi(\xi_2, \eta_2) \rb$ is maximal again need not be considered.
\end{proof}
Thus the quality of our well-posedness result depends entirely on the exponent in the number theoretic estimate~\eqref{generalNTestimate} that we use. The previously mentioned result due to Schmidt~\cite{Schmidt}*{Theorem~1} gives
\begin{kor}
	In Proposition~\ref{generalPerEstimate} one can choose $\gamma = \frac{3}{5}$ and thus Theorem~\ref{NVper} holds.
\end{kor}

\section{Open questions}

Unfortunately there are several questions that we cannot answer. They are immediately connected with our results here:

\begin{itemize}
 \item[(1)] Optimality in the nonperiodic case: Is the Cauchy problem for (NV)
    locally well-posed in $H^s(\R^2)$ for $s \in [-1,-\frac34]$? For KdV on the
    real line this gap was closed by the celebrated global $H^{-1}(\R)$-result
    of Killip and Vișan~\cite{KV}, but they had to go beyond iterative methods
    because KdV in $H^s(\R)$ is ill-posed for $s<-\frac34$ in the $C^0$-uniform
    sense by~\cite{KPV01}*{Theorem~1.4}. The problem with (NV) is possibly on a
    much lower level, since our attempt to prove $C^2$-illposedness below
    $H^{-\frac34}(\R^2)$ failed. Schottdorf's $L^2(\R^2)$-result for (mNV) in
    combination with the Miura-type map suggests in a sense, that one should be
    able to do the step down to $H^{-1}(\R^2)$ by the contraction mapping
    principle.
 \item[(2)] Optimality in the periodic case: Is the initial value problem for
    (NV) locally well-posed in $H^s_0(\T^2)$ for $s \in [-\frac12, -\frac15]$? In
    our proof we inserted the estimate
    $$\# (\Z^2 \cap H \cap Q_N) \le c N^{\frac35+}$$
    for the number of lattice points on a nondegenerate curve of type (i) and (ii) $H$ in a square
    $Q_N$ of size $N$. This estimate due to Schmidt~\cite{Schmidt} has the advantage
    of being independent of the shape of the curves. There are some estimates
    in the number theoretic literature with smaller exponents
    (e.g.~\cites{BP, Huxley}), which are valid for general sufficiently smooth
    curves, but it seems to be quite cumbersome to check whether they give the
    necessary \emph{uniform} bounds. Moreover, to get anything better than
    $\lesssim N^{\frac{4}{15}}$ seems to rely on specific properties of the
    family of curves in our considerations. Observe that an estimate
    $\lesssim N^{0+}$ for the number of lattice points would imply LWP in
    $H^s_0(\T^2)$ for $s > - \frac12$. Below $- \frac12$ there is
    $C^2$-illposedness by Bourgain's counterexample for KdV in the periodic
    case, see~\cite{Bmeasures}.
  \item[(3)] Can our result in the periodic case (valid for data of mean zero)
    be generalized to data of arbitrary mean? For KdV the reduction of the
    general to the mean zero case~\cite{B1}*{p.~219} is trivial in the sense that
    it leaves the $L^4$-estimate and the resonance function unchanged. For (NV)
    this reduction produces the additional linear term
    $$3\phi_0(\partial^2\overline{\partial}^{-1}+\overline{\partial}^2\partial^{-1})u, \qquad \mbox{where}\qquad \phi_0=\frac{1}{4\pi^2}\int_{\T^2}u_0(x,y)dxdy,$$
    which changes the phase function into
    $$\widetilde{\varphi}(\xi,\eta)= \varphi (\xi,\eta)(1+\frac{3\phi_0}{\xi^2+\eta^2}).$$
    With $E=3\phi_0$ this is precisely the situation of the ``nonzero energy''
    (NV) analyzed in~\cites{KMI, KMII} in the nonperiodic case. The resonance
    function is then disturbed by the additional term and the exact cancellation
    of the Fourier multiplier is destroyed.
\end{itemize}

\bibliography{refs.bib}

\begin{bibdiv}
\begin{biblist}

\bib{Angel}{article}{
      author={Angelopoulos, Yannis},
       title={{Well-posedness and ill-posedness results for the Novikov-Veselov
  equation}},
        date={2016-02},
        ISSN={1534-0392},
     journal={Communications on Pure and Applied Analysis},
      volume={15},
      number={3},
       pages={727\ndash 760},
  url={http://www.aimsciences.org/journals/displayArticlesnew.jsp?paperID=12244},
}

\bib{BKS}{article}{
      author={Ben-Artzi, Matania},
      author={Koch, Herbert},
      author={Saut, Jean-Claude},
       title={{Dispersion Estimates for Third Order Equations in Two
  Dimensions}},
        date={2003-01},
        ISSN={0360-5302, 1532-4133},
     journal={Communications in Partial Differential Equations},
      volume={28},
      number={11-12},
       pages={1943\ndash 1974},
}

\bib{Bog}{article}{
      author={Bogdanov, L.~V.},
       title={{The Veselov-Novikov equation as a natural generalization of the
  Korteweg-de Vries equation}},
    language={Russian},
        date={1987},
        ISSN={0564-6162},
     journal={Akademiya Nauk SSSR. Teoreticheskaya i Matematicheskaya Fizika},
      volume={70},
      number={2},
       pages={309\ndash 314},
  url={http://www.mathnet.ru/php/getFT.phtml?jrnid=tmf&paperid=4647&what=fullt&option_lang=eng},
        note={translation in~\cite{BogEN}},
}

\bib{BogEN}{article}{
      author={Bogdanov, L.~V.},
       title={{Veselov-Novikov equation as a natural two-dimensional
  generalization of the Korteweg-de Vries equation}},
        date={1987-02},
     journal={Theoretical and Mathematical Physics},
      volume={70},
      number={2},
       pages={219\ndash 223},
}

\bib{BLMP86}{article}{
      author={Boiti, M.},
      author={Leon, J. J.-P.},
      author={Manna, M.},
      author={Pempinelli, F.},
       title={{On the spectral transform of a Korteweg-de Vries equation in two
  spatial dimensions}},
        date={1986-08},
        ISSN={0266-5611, 1361-6420},
     journal={Inverse Problems},
      volume={2},
      number={3},
       pages={271\ndash 279},
}

\bib{BLMP87}{article}{
      author={Boiti, M.},
      author={Leon, J. J.-P.},
      author={Manna, M.},
      author={Pempinelli, F.},
       title={{On a spectral transform of a KdV-like equation related to the
  Schr{\"{o}}dinger operator in the plane}},
        date={1987-02},
        ISSN={0266-5611, 1361-6420},
     journal={Inverse Problems},
      volume={3},
      number={1},
       pages={25\ndash 36},
}

\bib{BP}{article}{
      author={Bombieri, E.},
      author={Pila, J.},
       title={The number of integral points on arcs and ovals},
        date={1989-10},
     journal={Duke Mathematical Journal},
      volume={59},
      number={2},
       pages={337\ndash 357},
}

\bib{B1}{article}{
      author={Bourgain, J.},
       title={{Fourier transform restriction phenomena for certain lattice
  subsets and applications to nonlinear evolution equations: Part I:
  Schr{\"{o}}dinger equations}},
        date={1993-03},
        ISSN={1016-443X, 1420-8970},
     journal={Geometric and Functional Analysis},
      volume={3},
      number={2},
       pages={107\ndash 156},
}

\bib{B2}{article}{
      author={Bourgain, J.},
       title={{Fourier transform restriction phenomena for certain lattice
  subsets and applications to nonlinear evolution equations: Part II: The
  KdV-equation}},
        date={1993-05},
        ISSN={1016-443X, 1420-8970},
     journal={Geometric and Functional Analysis},
      volume={3},
      number={3},
       pages={209\ndash 262},
}

\bib{Bmeasures}{article}{
      author={Bourgain, J.},
       title={{Periodic Korteweg de Vries equation with measures as initial
  data}},
        date={1997-03},
        ISSN={1022-1824, 1420-9020},
     journal={Selecta Mathematica},
      volume={3},
      number={2},
       pages={115\ndash 159},
}

\bib{CMMPSS}{article}{
      author={Croke, R.},
      author={Mueller, J.},
      author={Music, M.},
      author={Perry, P.},
      author={Siltanen, S.},
      author={Stahel, A.},
       title={{The Novikov-Veselov Equation: Theory and Computation}},
        date={2015},
        ISSN={978-1-4704-1050-6 978-1-4704-2274-5},
     journal={Contemporary Mathematics},
      volume={635},
       pages={25\ndash 70},
         url={http://www.ams.org/conm/635},
}

\bib{CMS}{article}{
      author={Croke, Ryan},
      author={Mueller, Jennifer},
      author={Stahel, Andreas},
       title={{Transverse instability of plane wave soliton solutions of the
  Novikov-Veselov equation}},
        date={2015},
        ISSN={978-1-4704-1050-6 978-1-4704-2274-5},
     journal={Contemporary Mathematics},
      volume={635},
       pages={71\ndash 89},
         url={http://www.ams.org/conm/635},
}

\bib{GTV}{article}{
      author={Ginibre, Jean},
      author={Tsutsumi, Yoshio},
      author={Velo, Giorgio},
       title={{On the Cauchy Problem for the Zakharov System}},
        date={1997},
        ISSN={0022-1236},
     journal={Journal of Functional Analysis},
      volume={151},
      number={2},
       pages={384\ndash 436},
         url={https://linkinghub.elsevier.com/retrieve/pii/S0022123697931487},
}

\bib{GV}{article}{
      author={Ginibre, Jean},
      author={Velo, Giorgio},
       title={{Generalized Strichartz Inequalities for the Wave Equation}},
        date={1995-10},
        ISSN={0022-1236},
     journal={Journal of Functional Analysis},
      volume={133},
      number={1},
       pages={50\ndash 68},
         url={https://linkinghub.elsevier.com/retrieve/pii/S0022123685711196},
}

\bib{GM}{article}{
      author={{Grinevich}, P.~G.},
      author={{Manakov}, S.~V.},
       title={{Inverse scattering problem for the two-dimensional Schr\"odinger
  operator, the \({\bar\partial}\)-method and nonlinear equations}},
        date={1986},
     journal={Funktsional'nyi Analiz i ego Prilozheniya},
      volume={20},
      number={2},
       pages={14\ndash 24},
  url={http://www.mathnet.ru/php/archive.phtml?wshow=paper&jrnid=faa&paperid=1268&option_lang=eng},
        note={translation in~\cite{GMen}},
}

\bib{GMen}{article}{
      author={{Grinevich}, P.~G.},
      author={{Manakov}, S.~V.},
       title={{Inverse scattering problem for the two-dimensional Schr\"odinger
  operator, the \({\bar\partial}\)-method and nonlinear equations}},
        date={1986},
        ISSN={0016-2663},
     journal={Functional Analysis and its Applications},
      volume={20},
      number={2},
       pages={94\ndash 103},
}

\bib{HHK}{article}{
      author={Hadac, Martin},
      author={Herr, Sebastian},
      author={Koch, Herbert},
       title={{Well-posedness and scattering for the KP-II equation in a
  critical space}},
        date={2009-05},
        ISSN={0294-1449},
     journal={Annales de l'Institut Henri Poincar{\'{e}} C, Analyse non
  lin{\'{e}}aire},
      volume={26},
      number={3},
       pages={917\ndash 941},
         url={https://linkinghub.elsevier.com/retrieve/pii/S0294144908000474},
}

\bib{Huxley}{article}{
      author={Huxley, Martin~N.},
       title={The integer points in a plane curve},
        date={2007-01},
     journal={Functiones et Approximatio Commentarii Mathematici},
      volume={37},
      number={1},
       pages={213\ndash 231},
         url={https://doi.org/10.7169/facm/1229618752},
}

\bib{KMI}{article}{
      author={Kazeykina, Anna},
      author={Mu{\~{n}}oz, Claudio},
       title={{Dispersive estimates for rational symbols and local
  well-posedness of the nonzero energy NV equation}},
        date={2016-03},
        ISSN={0022-1236},
     journal={Journal of Functional Analysis},
      volume={270},
      number={5},
       pages={1744\ndash 1791},
         url={https://linkinghub.elsevier.com/retrieve/pii/S0022123615004929},
}

\bib{KMII}{article}{
      author={Kazeykina, Anna},
      author={Mu{\~{n}}oz, Claudio},
       title={{Dispersive estimates for rational symbols and local
  well-posedness of the nonzero energy NV equation. II}},
        date={2018},
        ISSN={0022-0396},
     journal={Journal of Differential Equations},
      volume={264},
      number={7},
       pages={4822\ndash 4888},
         url={https://linkinghub.elsevier.com/retrieve/pii/S0022039617306599},
}

\bib{KPV96a}{article}{
      author={Kenig, Carlos~E.},
      author={Ponce, Gustavo},
      author={Vega, Luis},
       title={{A bilinear estimate with applications to the KdV equation}},
        date={1996},
        ISSN={0894-0347, 1088-6834},
     journal={Journal of the American Mathematical Society},
      volume={9},
      number={2},
       pages={573\ndash 603},
         url={http://www.ams.org/jams/1996-9-02/S0894-0347-96-00200-7/},
}

\bib{KPV96b}{article}{
      author={Kenig, Carlos~E.},
      author={Ponce, Gustavo},
      author={Vega, Luis},
       title={{Quadratic forms for the 1-D semilinear Schr{\"{o}}dinger
  equation}},
        date={1996},
        ISSN={0002-9947, 1088-6850},
     journal={Transactions of the American Mathematical Society},
      volume={348},
      number={8},
       pages={3323\ndash 3353},
         url={http://www.ams.org/tran/1996-348-08/S0002-9947-96-01645-5/},
}

\bib{KPV01}{article}{
      author={Kenig, Carlos~E.},
      author={Ponce, Gustavo},
      author={Vega, Luis},
       title={{On the ill-posedness of some canonical dispersive equations}},
        date={2001-02},
        ISSN={0012-7094},
     journal={Duke Mathematical Journal},
      volume={106},
      number={3},
       pages={617\ndash 633},
         url={https://projecteuclid.org/euclid.dmj/1092403945},
}

\bib{KV}{article}{
      author={Killip, Rowan},
      author={Vi{\c{s}}an, Monica},
       title={{KdV is well-posed in $H^{-1}$}},
        date={2019},
        ISSN={0003-486X},
     journal={Annals of Mathematics},
      volume={190},
      number={1},
       pages={249\ndash 305},
}

\bib{KT05}{article}{
      author={Koch, Herbert},
      author={Tataru, Daniel},
       title={{Dispersive estimates for principally normal pseudodifferential
  operators}},
        date={2005},
        ISSN={0010-3640, 1097-0312},
     journal={Communications on Pure and Applied Mathematics},
      volume={58},
      number={2},
       pages={217\ndash 284},
}

\bib{KT07}{article}{
      author={Koch, Herbert},
      author={Tataru, Daniel},
       title={{A Priori Bounds for the 1D Cubic NLS in Negative Sobolev
  Spaces}},
        date={2007},
        ISSN={1687-0247, 1073-7928},
     journal={International Mathematics Research Notices},
      volume={2007},
}

\bib{LMSSII}{article}{
      author={Lassas, M.},
      author={Mueller, J.~L.},
      author={Siltanen, S.},
      author={Stahel, A.},
       title={{The Novikov-Veselov equation and the inverse scattering method:
  II. Computation}},
        date={2012},
        ISSN={0951-7715, 1361-6544},
     journal={Nonlinearity},
      volume={25},
      number={6},
       pages={1799\ndash 1818},
}

\bib{LMSSI}{article}{
      author={Lassas, M.},
      author={Mueller, J.~L.},
      author={Siltanen, S.},
      author={Stahel, A.},
       title={{The Novikov-Veselov equation and the inverse scattering method,
  Part I: Analysis}},
        date={2012},
        ISSN={0167-2789},
     journal={Physica D: Nonlinear Phenomena},
      volume={241},
      number={16},
       pages={1322\ndash 1335},
         url={https://linkinghub.elsevier.com/retrieve/pii/S0167278912001212},
}

\bib{MS}{article}{
      author={Montgomery-Smith, S.~J.},
       title={{Time decay for the bounded mean oscillation of solutions of the
  Schrödinger and wave equations}},
        date={1998},
     journal={Duke Mathematical Journal},
      volume={91},
      number={2},
       pages={393\ndash 408},
         url={https://doi.org/10.1215/S0012-7094-98-09117-7},
}

\bib{M}{article}{
      author={Music, Michael},
       title={{The nonlinear Fourier transform for two-dimensional subcritical
  potentials}},
        date={2014},
        ISSN={1930-8345},
     journal={Inverse Problems \& Imaging},
      volume={8},
      number={4},
       pages={1151\ndash 1167},
}

\bib{MP}{article}{
      author={Music, Michael},
      author={Perry, Peter},
       title={{Global Solutions for the zero-energy Novikov-Veselov equation by
  inverse scattering}},
        date={2018},
        ISSN={0951-7715, 1361-6544},
     journal={Nonlinearity},
      volume={31},
      number={7},
       pages={3413\ndash 3440},
}

\bib{NV86}{article}{
      author={Novikov, S.~P.},
      author={Veselov, A.~P.},
       title={{Two-dimensional Schr{\"{o}}dinger operator: Inverse scattering
  transform and evolutional equations}},
        date={1986},
        ISSN={0167-2789},
     journal={Physica D: Nonlinear Phenomena},
      volume={18},
      number={1-3},
       pages={267\ndash 273},
         url={https://linkinghub.elsevier.com/retrieve/pii/0167278986901879},
}

\bib{Perry}{article}{
      author={Perry, Peter},
       title={{Miura maps and inverse scattering for the Novikov-Veselov
  equation}},
        date={2014},
        ISSN={1948-206X, 2157-5045},
     journal={Analysis \& PDE},
      volume={7},
      number={2},
       pages={311\ndash 343},
         url={http://msp.org/apde/2014/7-2/p02.xhtml},
}

\bib{Schmidt}{article}{
      author={Schmidt, Wolfgang~M.},
       title={{Integer points on curves and surfaces}},
        date={1985},
        ISSN={0026-9255, 1436-5081},
     journal={Monatshefte für Mathematik},
      volume={99},
      number={1},
       pages={45\ndash 72},
}

\bib{Schott}{thesis}{
      author={Schottdorf, Tobias},
       title={{Global Existence Without Decay}},
        type={Ph.D. Thesis},
        date={2013},
         url={https://hdl.handle.net/20.500.11811/5801},
        note={online: https://hdl.handle.net/20.500.11811/5801},
}

\bib{TTen}{article}{
      author={Taimanov, I.~A.},
      author={Tsarev, S.~P.},
       title={Two-dimensional rational solitons and their blowup via the
  moutard transformation},
    language={en},
        date={2008-11},
        ISSN={0040-5779, 1573-9333},
     journal={Theoretical and Mathematical Physics},
      volume={157},
      number={2},
       pages={1525\ndash 1541},
         url={http://link.springer.com/10.1007/s11232-008-0127-3},
}

\bib{TT}{article}{
      author={Taimanov, Iskander~Asanovich},
      author={Tsarev, Sergey~Petrovich},
       title={Two-dimensional rational solitons constructed using the {Moutard}
  transformations and their decay},
        date={2008},
        ISSN={0564-6162},
     journal={Teoreticheskaya i Matematicheskaya Fizika},
      volume={157},
      number={2},
       pages={188\ndash 207},
         url={http://mi.mathnet.ru/tmf6274},
        note={translation in~\cite{TTen}},
}

\bib{T}{article}{
      author={Terence, Tao},
       title={A counterexample to an endpoint bilinear strichartz inequality},
        date={200612},
     journal={Electronic Journal of Differential Equations},
      volume={2006},
       pages={1\ndash 6},
}

\bib{T93}{article}{
      author={Tsai, T.-Y.},
       title={{The Schrödinger operator in the plane}},
        date={1993-12},
        ISSN={0266-5611, 1361-6420},
     journal={Inverse Problems},
      volume={9},
      number={6},
       pages={763\ndash 787},
}

\bib{NV84A}{article}{
      author={Veselov, A.~P.},
      author={Novikov, S.~P.},
       title={{Finite-gap two-dimensional potential Schr{\"{o}}dinger
  operators. Explicit formulas and evolution equations}},
    language={Russian},
        date={1984},
        ISSN={0002-3264},
     journal={Doklady Akademii Nauk SSSR},
      volume={279},
      number={1},
       pages={20\ndash 24},
}

\bib{NV84B}{article}{
      author={Veselov, A.~P.},
      author={Novikov, S.~P.},
       title={{Finite-gap two-dimensional Schr{\"{o}}dinger operators.
  Potential operators}},
    language={Russian},
        date={1984},
        ISSN={0002-3264},
     journal={Doklady Akademii Nauk SSSR},
      volume={279},
      number={4},
       pages={784\ndash 788},
}

\end{biblist}
\end{bibdiv}
\end{document}